\documentclass[preprint,12pt]{elsarticle}
\usepackage{amsmath,amssymb,amsthm}
\usepackage{graphicx}
\usepackage{color}
\biboptions{sort&compress}

\makeatletter
\def\ps@pprintTitle{%
  \let\@oddhead\@empty
  \let\@evenhead\@empty
  \def\@oddfoot{\reset@font\hfil\thepage\hfil}
  \let\@evenfoot\@oddfoot}
\makeatother

\newtheorem{lettertheorem}{Theorem}

\newtheorem{defin}{Definition}[section]

\newtheorem{theorem}[defin]{Theorem}
\newtheorem{lemma}[defin]{Lemma}

\newtheorem{qu}{Question}
\newtheorem{exa}[defin]{Example}
\newenvironment{example}{\begin{exa}\rm}{\end{exa}}

\newtheorem{remark}[]{Remark}
\begin{document}

\begin{frontmatter}

\title{On value distribution of certain delay-differential polynomials}
%\tnotetext[mytitlenote]{This work was supported by
%grant ZR2016AQ20 from the NSF of Shandong Province, grant 11626112 from the NSFC %Tianyuan Mathematics Youth Fund, the NNSF of China No.11371225,
%and grant XBS1630 from the Fund of Doctoral Program Research of University of %Jinan.}

\author[mymainaddress]{Nan Li\corref{cor1}} %\fnref{fn1}
\ead{nanli32787310@163.com}

\author[mysecondaryaddress]{Lianzhong Yang}%\fnref{fn1}\fnref{fn2}
\ead{lzyang@sdu.edu.cn}

\cortext[cor1]{Corresponding author: Nan Li}

\address[mymainaddress]{School of Mathematics, Qilu Normal University, Jinan, Shandong, 250200, P.R.China}

\address[mysecondaryaddress]{School of Mathematics, Shandong University,
  Jinan, Shandong, 250100, P.R.China}

\begin{abstract}
Given an entire function $f$ of finite order $\rho$, let
$L(z,f)=\sum_{j=0}^{m}b_{j}(z)f^{(k_{j})}(z+c_{j})$ be a linear
delay-differential polynomial of $f$ with small  coefficients in the sense
of $O(r^{\lambda+\varepsilon})+S(r,f)$, $\lambda<\rho$.
Provided $\alpha$, $\beta$ be similar small functions, we consider the
zero distribution of $L(z,f)-\alpha f^{n}-\beta$ for $n\geq 3$ and
$n=2$, respectively. Our results are improvements and complements of  Chen(Abstract  Appl. Anal., 2011, 2011: ID239853, 1--9),
 and Laine (J. Math. Anal. Appl. 2019, 469(2):  808--826.), etc.
\end{abstract}

\end{frontmatter}

\section{Introduction}

Let $f(z)$ be a transcendental meromorphic function in the complex plane
$\mathbb{C}$.  We assume that the reader is familiar with the standard notations and main results in Nevanlinna theory (see \cite{Hayman},\cite{Laine},\cite{Yi1}). A meromorphic function $\alpha$ is
said to be a $\lambda$-small function of a meromorphic function $f$ of
finite order  $\rho$, if there exists $\lambda<\rho$, such that
for any $\varepsilon\in (0,\rho-\lambda)$,
\begin{eqnarray}\label{introeq2aaa}
  T(r,\alpha)=O\left(r^{\lambda+\varepsilon} \right)+S(r,f),
\end{eqnarray}
outside a possible exceptional set $F$ of finite logarithmic measure.
Here, $S(r,f)$ is any quantity that satisfies $S(r,f)=o(T(r,f))$ as
$r\to\infty$ outside a set $F$. For the sake of simplicity, the right
hand side in \eqref{introeq2aaa} will be denoted by $S_{\lambda}(r,f)$.

Hayman \cite{Hayman2}   proved the following theorem.
\begin{lettertheorem}\label{thmAH}
If $f(z)$ is a transcendental entire function, $n\geq 3$ is an integer
and $a(\neq 0)$ is a constant, then $f'(z)-af(z)^{n}$ assumes all finite
values infinitely often.
\end{lettertheorem}

Recently, several articles (see \cite{chen2011, chen2,chen-gao-zhang, Chiang-Feng, Halburd-Korhonen, Korhonen2009,   Laine2019, Laine-Latreuch, Liu-Laine, Liu-Laine-Yang, Liu-Yi,Liu-Zhang-Qi,Long-Wu-Zhu, Zheng-Chen} etc.) have focused on complex differences, giving many
difference analogues in value distribution theory of meromorphic functions.

In 2011,  Chen \cite{chen2} obtained the following Theorem~\ref{thmAJ},  an almost direct difference analogue of Theorem~\ref{thmAH}, and gave an estimate of numbers of $b$-points, namely,
$\lambda(\Psi_{n}(z)-b)=\sigma(f)$ for every $b\in \mathbb{C}$.
.

\begin{lettertheorem}\label{thmAJ}
Let $f(z)$ be a transcendental entire function of finite order, and let
$\alpha,c\in \mathbb{C}\setminus\{0\}$ be constants, with $c$ such that $f(z+c)\not\equiv f(z)$. Set $\Psi_{n}(z)=\Delta f(z)-\alpha f(z)^{n}$, where
$\Delta f(z)=f(z+c)-f(z)$ and $n\geq 3$ is an integer. Then
$\Psi_{n}(z)$ assumes all finite values infinitely often, and for every $\beta\in \mathbb{C}$, we have $\lambda(\Psi_{n}(z)-\beta)=\sigma(f)$.
\end{lettertheorem}

In 2013, Liu and Yi \cite{Liu-Yi} replaced $\Delta f(z)$ in Theorem~\ref{thmAJ} by a more general linear difference operator $g(f)=\sum_{j=1}^{k}a_{j}f(z+c_{j})$,
where $a_{j},c_{j}(j=1,2,\ldots,k)$ are complex constants,
and obtained the following result.
\begin{lettertheorem}\label{thmAC0}
Let $f(z)$ be a transcendental entire function of finite order $\rho(f)$, let
$\alpha,\beta$ be complex constants. Set $\Psi_{n}=g(f)-\alpha f^{n}(z)$, where
$n\geq 3$ is an integer. Then $\Psi_{n}$ have infinitely many zeros and
$\lambda(\Psi_{n}-\beta)=\rho(f)$ provided that $g(f)\not\equiv \beta$.
\end{lettertheorem}

In 2019, Laine \cite{Laine}  generalized the coefficients
from complex constants to $\lambda$-small functions,
released the assumption on $\beta$ that $g(f)\not\equiv \beta$, and obtained the following theorem.

\begin{lettertheorem}\label{thmAC0000}
Let $f$ be an entire function of finite order $\rho(f)$, $\alpha$,
$\beta$, $b_{0},\ldots, b_{k}$ be $\lambda$-small functions of $f$, $g(f):=\Sigma_{j=1}^{k}b_{j}(z)f(z+c_{j})$ be non-vanishing and
$n\geq 3$. Then for $\Psi_{n}:=g(f)-\alpha f^{n}$, $\Psi_{n}-\beta$ has
sufficiently many zeros to satisfy $\lambda(\Psi_{n}-\beta)=\rho(f)$.
\end{lettertheorem}

But a bit regret, the proof of dealing with $G(z,f)\equiv 0$ in  \cite[Theorem 5.1]{Laine}  is not  complete.

%In the same year, Liu, Zhang and Qi \cite{Liu-Zhang-Qi} generalized the %results in Theorem~\ref{thmAC0}  to meromorphic functions, and obtained %$\lambda(\Psi_{n}-\beta)=\rho(f)$
% with the assumption that $\lambda(1/f)<\rho(f)$ and $g(f)\not\equiv \beta$,
%where $\beta$ is a small function of $f$.

We now introduce the generalized linear delay-differential operator of
$f(z)$,
\begin{eqnarray}\label{thmABABnew2eq1add1}
  L(z,f)=\sum_{j=0}^{m}b_{j}(z)f^{(k_{j})}(z+c_{j}),
\end{eqnarray}
where $b_{j}$ are  $\lambda$-small functions of $f$, $c_{j}$ are distinct
complex numbers and $k_{j}$ are non-negative integers. In view of the
above theorems, it is quiet natural to study the value distribution of
 $\Psi_{n}-\beta$ when the linear difference operator $g(f)$ is changed to
 the linear delay-differential operator $L(z,f)$ with the restriction
 on $\beta$ be omitted?

In this paper, we study the above problem and obtain the following result.
\begin{theorem}\label{thmABABnew2}
Let $f(z)$ be an entire function of finite order $\rho$, $\alpha (\not\equiv 0)$, $\beta$ be $\lambda$-small functions of $f$, $L(z,f)$ be non-vanishing  linear  delay-differential polynomial defined as in \eqref{thmABABnew2eq1add1} and
$n\geq 3$. Then for $\Phi_{n}=L(z,f)-\alpha f^{n}$, we have $\Phi_{n}-\beta$  has sufficiently many zeros to satisfy $\lambda(\Phi_{n}-\beta)=\rho(f)$.
\end{theorem}

\begin{remark}
We omit the restriction on $\beta$ is meaningful. In fact,
we do not need to worry about that if $L(z,f)\equiv \beta$ and
$f$ has a Borel exceptional value $0$, then $\lambda(\Phi_{n}-\beta)
=\lambda(-\alpha f^{n})$ may less than $\rho$. It is because from the proof of Therorem~\ref{thmABABnew2},
we can get that if  $0$ is a Borel exceptional value of $f$, then
$L(z,f)\not\equiv \beta$. So $L(z,f)\equiv \beta$ and
$f$ has a Borel exceptional value $0$ can not hold simultaneously.
\end{remark}

Chen  \cite{chen2} also considered the value distribution of  $\Psi_{2}$ when $n=2$ and obtained the following Theorems~\ref{thmAK} and ~\ref{thmAL}.

\begin{lettertheorem}\label{thmAK}
Let $f(z)$ be a transcendental entire function of finite order with
a Borel exceptional value $0$, and let $\alpha,c\in\mathbb{C}\setminus \{0\}$
be constants, with $c$ such that $f(z+c)\not\equiv f(z)$. Then $\Psi_{2}(z)$
assumes all finite values infinitely often, and for every $\beta\in \mathbb{C}$
we have $\lambda(\Psi_{2}-\beta)=\sigma(f)$.
\end{lettertheorem}

\begin{lettertheorem}\label{thmAL}
Let $f(z)$ be a transcendental entire function of finite order with
a finite nonzero Borel exceptional value $d$,  and let $\alpha,c\in\mathbb{C}\setminus \{0\}$ be constants, with $c$ such that $f(z+c)\not\equiv f(z)$. Then  for every $\beta\in \mathbb{C}$ with $\beta\neq-\alpha d^{2}$, $\Psi_{2}(z)$ assumes the value $\beta$ infinitely often, and $\lambda(\Psi_{2}-\beta)=\sigma(f)$.
\end{lettertheorem}

Liu and Yi \cite{Liu-Yi} replaced $\Delta f(z)$ in Theorems~\ref{thmAK}
and \ref{thmAK} to a more general linear difference operator $\sum_{j=1}^{k} a_{j}(z)f(z+c_{j})$ and
obtained the following result.
\begin{lettertheorem}\label{thmBZ}
Suppose that $f(z)$ be a finite order transcendental entire function with
a Borel exceptional value $d$. Let $\beta(z), \alpha(z)(\not\equiv 0), a_{j}(z)(j=1,2,\ldots,k)$ be polynomials, and let $c_{j}(j=1,2,\ldots,k)$
be complex constants. If either $d=0$ and  $\sum_{j=1}^{k} a_{j}(z)f(z+c_{j})\not\equiv 0$, or  $d\neq 0$ and  $\sum_{j=1}^{k} d a_{j}(z)-d^{2}\alpha(z)-\beta(z)\not\equiv 0$, then $\Psi_{2}(z)-\beta(z)=
\sum_{j=1}^{k}a_{j}(z)f(z+c_{j})-\alpha(z)f(z)^{2}-\beta(z)$ has infinitely many
zeros and $\lambda(\Psi_{2}-\beta)=\rho(f)$.
\end{lettertheorem}

The following Example~\ref{examAA} shows that if the difference operator $\Delta f(z)=f(z+c)-f(z)$ or  $\sum_{j=1}^{k}a_{j}(z)f(z+c_{j})$ in $\Psi_{2}$  is changed to a linear delay-differential operator $L(z,f)$,  the conclusions in Theorems~\ref{thmAL}
 and~\ref{thmBZ} may not hold.

\begin{example}\label{examAA}
Let $L(z,f)=f(z+1)-f'(z)$, and $\Phi_{2}=L(z,f)-\frac{e-1}{2} f(z)^{2}$.
For $f_{1}(z)=e^{z}+1$, we have $\Phi_{2}(f_{1})
=\frac{1-e}{2}e^{2z}+\frac{3-e}{2}$. Here, $d=1$, $\alpha=\frac{e-1}{2}$, $a_{1}=1$, $a_{2}=-1$,
and $\beta=\frac{3-e}{2}\neq \sum_{j=1}^{k} d a_{j}-\alpha d^{2}=-\frac{e-1}{2}$, but
$\Phi_{2}\neq \beta$.
\end{example}

So it is natural to ask:  what can we say about $\Phi_{2}=L(z,f)-\alpha f^{2}$ ?
The second aim of this paper is to consider the above problem, and obtain the following results.

Before stating Theorem~\ref{thmABABnew4}, we
recall that the Borel exceptional value for small function $\beta$ of $f(z)$
satisfies
\begin{eqnarray*}
  \lambda(f(z)-\beta)<\rho(f),
\end{eqnarray*}
where $\lambda(f-\beta)$ is the exponent of convergence of zeros of
$f-\beta$.
%and we say $d$ is ``finite"  if $d\not\equiv \infty$.

\begin{theorem}\label{thmABABnew4}
Let $f(z)$ be a transcendental entire function of finite order $\rho$ with a finite non-zero Borel exceptional value $d$.  Let $\alpha\in\mathbb{C}\setminus \{0\}$ be constant,  and $\beta, \, b_{j}\, (j=0,1,\ldots,m)$ be  $\lambda$-small entire functions of $f$.
Let $L(z,f)$ be non-vanishing  linear  delay-differential polynomial defined as in \eqref{thmABABnew2eq1add1}. Defining $\Phi_{2}=L(z,f)-\alpha f^{2}$, and $I_{1}=\{0\leq j\leq m: k_{j}=0\}$, we have the following statements:
\begin{itemize}
 \item [(i)]   If
\begin{eqnarray*}
\beta\not\equiv\left(\sum_{j\in I_{1}}b_{j}\right) d-\alpha d^{2},
\end{eqnarray*}
then $\Phi_{2}(z)-\beta$ has sufficiently many zeros to satisfy
$\lambda(\Phi_{2}-\beta)=\rho$.
%\begin{eqnarray*}
%d^{*} \equiv\left(\sum_{j\in I_{1}}b_{j}\right) d-\alpha d^{2},
%\end{eqnarray*}

 \item [(ii)] If
\begin{eqnarray}\label{introeq2}
\beta \equiv\left(\sum_{j\in I_{1}}b_{j}\right) d-\alpha d^{2},
\end{eqnarray}
 then one of the following holds:
\begin{itemize}
  \item [(a)] $\beta$ is a  Borel exceptional small function of $\Phi_{2}$, which satisfies
\begin{eqnarray}\label{introeq1aaa}
   \frac{\beta-\Phi_{2}}{(f-d)^{2}}=\alpha=\frac{ L(z,f)-\alpha d^{2}-\beta}{2d(f-d)}.
\end{eqnarray}
 \item [(b)] $\Phi_{2}-\beta$ has sufficiently many zeros to satisfy
\begin{eqnarray*}
 N\left(r,\frac{1}{\Phi_{2}-\beta}\right)=
T(r,f)+S_{\lambda}(r,f).
\end{eqnarray*}
\end{itemize}

\end{itemize}
\end{theorem}

\begin{remark}
In Example~\ref{examAA},
 $\sum_{j\in I_{1}}b_{j}=1$, $d=1$, $\alpha=\frac{e-1}{2}$, and  $\beta=\frac{3-e}{2}=\left(\sum_{j\in I_{1}}b_{j}\right) d-\alpha d^{2}$
is a  Borel exceptional value of $\Phi_{2}$, which also satisfies Eq.~\eqref{introeq1aaa}. Thus Example~\ref{examAA} above illustrates Theorem~\ref{thmABABnew4}.
\end{remark}

\begin{remark}
Let $L_{1}(z,f)=f(z+c)-f(z)$, then $\sum_{j\in I_{1}}b_{j}=0$.
Let $L_{2}(z,f)=\sum_{j=1}^{k} a_{j}(z)f(z+c_{j})$, then
$\sum_{j=1}^{k} a_{j}=\sum_{j\in I_{1}}a_{j}$. Thus by Theorem~\ref{thmABABnew4}\, (i), we can obtain the results in Theorems~\ref{thmAL} and~\ref{thmBZ} when $d\neq 0$.
Therefore Theorem~\ref{thmABABnew4} improves Theorems~\ref{thmAL}
and~\ref{thmBZ}.
\end{remark}

The following theorem deals with the case when $d=0$.

\begin{theorem}\label{thmABABnew7}
Let $f(z)$ be a transcendental entire function of finite order $\rho$ with
a Borel exceptional value $0$. Let $\alpha\in\mathbb{C}\setminus \{0\}$ be constant,  and $\beta, \, b_{j}\, (j=0,1,\ldots,m)$  be  $\lambda$-small entire functions of $f$.
Let $L(z,f)$  be non-vanishing  linear  delay-differential polynomial defined as in \eqref{thmABABnew2eq1add1}.  Defining $\Phi_{2}=L(z,f)-\alpha f^{2}$, then we have $\lambda(\Phi_{2}-\beta)=\rho$. Particularly, if $\beta \equiv 0$,
then
\begin{eqnarray}\label{thmABABnew7eq1}
 N\left(r,\frac{1}{\Phi_{2}}\right)=
T(r,f)+S_{\lambda}(r,f).
\end{eqnarray}

\end{theorem}

\begin{remark}
 The following example
shows that when $\beta\equiv0$, Eq.\eqref{thmABABnew7eq1} in  Theorem~\ref{thmABABnew7} occurs.
\end{remark}

\begin{example}\label{examAD}
Let $L(z,f)=e^{-2}f(z+1)-\frac{1}{2}f'(z)$, and $\Phi_{2}=L(z,f)-f(z)^{2}$.
For $f_{2}(z)=e^{z+1}$, we have $\Phi_{2}(f_{2})=(1-\frac{e}{2})e^{z}-e^{2}e^{2z}$. Here, $0$ is a Borel exceptional value of $f_{2}$, and $N\left(r,\frac{1}{\Phi_{2}}\right)=N\left(r, \frac{1}{1-\frac{e}{2}-e^{2}e^{z}}\right)=
T(r,f_{2})+S(r,f_{2})$.
\end{example}

\section{Preliminary Lemmas}

In this section, we collect the results that are needed for proving the main
results.

The following lemma plays an important role in
uniqueness problems of meromorphic functions.

%\noindent
\begin{lemma}[\cite{Yi1}]\label{lemma21}
 Let $f_{j}(z)\, (j=1,\ldots,n)\, (n\geq2)$ be meromorphic
functions, and let $g_{j}(z)\, (j=1,\ldots,n)$ be entire functions satisfying
\begin{itemize}
\item[(i)] $\sum _{j=1}^{n}f_{j}(z)e^{g_{j}(z)}\equiv 0;$
\item[(ii)]when $1\leq j < k \leq n,$ then $g_{j}(z)-g_{k}(z)$ is not a constant;
\item[(iii)] when $1\leq j \leq n, 1\leq h < k \leq n$, then
$$T(r,f_{j})=o \{T(r,e^{g_{h}-g_{k}})\} \quad (r\to \infty, r\not\in E),$$
where $E\subset(1,\infty)$ is of finite linear measure or logarithmic measure.
\end{itemize}
Then, $f_{j}(z)\equiv 0\, (j=1,\ldots ,n)$.
\end{lemma}

Using the same reasoning as in the proof of \cite[Lemma 2.4.2]{Laine}, we easily get the following lemma.

\begin{lemma}\cite{Laine-Latreuch}\label{lemma24}
Let $f$ be a transcendental meromorphic solution of finite order $\rho$ of
a differential-difference equation:
 \begin{equation*}
   f^{n}P(z,f)=Q(z,f),
 \end{equation*}
where $P(z,f)$ and $Q(z,f)$ are delay-differential polynomials in $f$
with $\lambda$-small coefficients of $f$. If the total degree of $Q(z,f)$
is $\leq n$, then for each $\varepsilon>0$,
\begin{eqnarray*}
m(r,P(z,f))=O(r^{\rho-1+\varepsilon})+S_{\lambda}(r,f).
\end{eqnarray*}
\end{lemma}

The following lemma, which is a special case of   \cite[Theorem 3.1]{Korhonen2009}, gives a relationship for the
Nevanlinna characteristic of a meromorphic function with its shift.
\begin{lemma}\cite{Korhonen2009}\label{lemmaxza}
Let $f(z)$ be a meromorphic function with the hyper-order less than one, and
$c\in\mathbb{C}\setminus \{0\}$. Then we have
\begin{eqnarray*}
 T(r,f(z+c))=T(r,f(z))+S(r,f).
\end{eqnarray*}
\end{lemma}
Observe that
\begin{eqnarray*}
 m\left(r,\frac{f^{(k)}(z+c)}{f(z)} \right)
\leq m\left(r,\frac{f^{(k)}(z+c)}{f(z+c)} \right)+m\left(r,\frac{f(z+c)}{f(z)} \right),
\end{eqnarray*}
by using Logarithmic Derivative Lemma and its difference analogues (see \cite{Chiang-Feng, Halburd-Korhonen,Korhonen2009,Laine}),
Lemma~\ref{lemmaxza}, we obtain the following lemma, see also \cite{Liu-Laine-Yang}.
\begin{lemma}\label{lemmaxzab}
Let $f$ be a transcendental meromorphic function of finite order.
Then
\begin{eqnarray}\label{lemmazabeq1}
  m \left(r,\frac{f^{(k)}(z+c)}{f(z)}\right)=S(r,f),
\end{eqnarray}
outside a possible exceptional set of finite logarithmic measure.
\end{lemma}

Applying Lemma~\ref{lemmaxzab}, we obtain the following lemma.
\begin{lemma}\label{lemma000}
Let $f$ be an entire function of finite order $\rho$, $\alpha (\not\equiv 0)$,
$\beta$ be $\lambda$-small functions of $f$, $L(z,f)$ be non-vanishing linear  delay-differential polynomial defined as in \eqref{thmABABnew2eq1add1} and
$n\geq 2$. Then $\Phi_{n}-\beta$ is transcendental and satisfying $\sigma(\Phi_{n}-\beta)=\rho$.
\end{lemma}
\begin{proof}
We first assume that $\Phi_{n}-\beta$ is transcendental. Indeed, if not,
then $\Phi_{n}-\beta=R(z)$ is rational, and $f^{n}=\alpha^{-1}(L(z,f)-\beta-R(z))$. Therefore, by Lemma~\ref{lemmaxzab},
we obtain
\begin{eqnarray*}
  nT(r,f)&=&T(r,f^{n})
  \leq T(r,L(z,f))+S_{\lambda}(r,f)\\
  &=& m(r,L(z,f))+S_{\lambda}(r,f) \\
  &\leq& m\left(r,\frac{L(z,f)}{f}\right)+m(r,f)+S_{\lambda}(r,f)\\
   &\leq& \sum_{j=0}^{m} m\left(r,  \frac{f^{(k_{j})}(z+c_{j})}{f(z)} \right)+m(r,f)+S_{\lambda}(r,f) \\
  &\leq& T(r,f)+S_{\lambda}(r,f),
\end{eqnarray*}
a contradiction follows since $n\geq 2$.

Next, we prove that $\sigma(\Phi_{n}-\beta)=\rho(f)$. By Lemma~\ref{lemmaxzab}, we have
\begin{eqnarray}\label{th3eq2lem}
  T(r,\Phi_{n}-\beta)&=&T\left(r,\left(L(z,f)-\alpha f^{n}-\beta  \right)  \right)\nonumber \\
  &\leq& T(r,f^{n})+T(r, L(z,f))+S_{\lambda}(r,f)\nonumber  \\
  &=& nT(r,f)+m(r,L(z,f))+S_{\lambda}(r,f) \nonumber\\
  &\leq& nT(r,f)+m\left(r,\frac{L(z,f)}{f} \right)+m(r,f)+S_{\lambda}(r,f) \nonumber\\
  &=& (n+1)T(r,f)+S_{\lambda}(r,f), %&=&(n+1+o(1))T(r,f)+O(r^{\lambda+\varepsilon}),
\end{eqnarray}
and
\begin{eqnarray}\label{th3eq3add2lem}
   T(r,\Phi_{n}-\beta)&=&T\left(r,\left(L(z,f)-\alpha f^{n}-\beta  \right)  \right)\nonumber \\
  &\geq& T(r,f^{n})-T(r, L(z,f))-S_{\lambda}(r,f)\nonumber  \\
&=& nT(r,f)-m(r,L(z,f))-S_{\lambda}(r,f) \nonumber\\
&\geq& nT(r,f)-m(r,f)-S_{\lambda}(r,f) \nonumber\\
 &=& (n-1)T(r,f)-S_{\lambda}(r,f).
\end{eqnarray}
Therefore, combining with $\lambda<\rho$, from \eqref{th3eq2lem} and
  \eqref{th3eq3add2lem}  we have $\sigma(\Phi_{n}-\beta)=\rho$.

\end{proof}

\section{Proof of Theorem~\ref{thmABABnew2}.}

Firstly, we  prove the case $\rho>0$. Suppose now, contrary to the assertion, that
$\lambda(\Phi_{n}-\beta)=\lambda<\rho$. From Lemma~\ref{lemma000}, we obtain that $\Phi_{n}-\beta$ is transcendental
and \eqref{th3eq3add2lem} holds.
By the
standard Hadamard representation, we may write
\begin{eqnarray}\label{th3eq1}
 \Phi_{n}-\beta= L(z,f)-\alpha f^{n}-\beta=\pi e^{g},
\end{eqnarray}
where $\pi$ is a non-vanishing $\lambda$-small function of
$f$, and $g$ is a polynomial with $\deg g\leq \rho$.
Actually, $\deg g=\rho$. Otherwise, if $\deg g\leq \mu<\rho$, then
from \eqref{th3eq3add2lem}  and \eqref{th3eq1}, we obtain
\begin{eqnarray*}
 (n-1)T(r,f)-S_{\lambda}(r,f)\leq  T(r,\Phi_{n}-\beta)=O(r^{\mu+\varepsilon})+S_{\lambda}(r,f),
\end{eqnarray*}
leading to $\rho\leq \max\{\mu,\lambda\}<\rho$ by $n\geq 3$, a contradiction.

Differentiating \eqref{th3eq1} and eliminating $e^{g}$, we obtain
\begin{eqnarray}\label{th3eq2}
  f(z)^{n-1}G(z,f)=H(z,f),
\end{eqnarray}
where
\begin{eqnarray*}
 G(z,f):=\left(\left(\frac{\pi'}{\pi}+g'\right)\alpha-\alpha'\right)f-n\alpha f'
\end{eqnarray*}
and
\begin{eqnarray*}
 H(z,f):=\left(\frac{\pi'}{\pi}+g'\right)L -L'
 -\left(\frac{\pi'}{\pi}+g'\right)\beta+\beta'.
\end{eqnarray*}

Case 1. $G(z, f)\equiv 0$.  Then we have $\alpha f^{n}=\widetilde{c}\pi e^{g}$
for some non-zero constant $\widetilde{c}$. By \eqref{th3eq1}, we get
\begin{eqnarray}\label{th3eq3}
  L-\beta=\left(\frac{1}{\widetilde{c}}+1 \right)\alpha f^{n}.
\end{eqnarray}

Subcase 1.1. $\widetilde{c}=-1$. Then we have $f=(-\pi/\alpha)^{1/n}e^{g/n}$ and $L\equiv\beta$. This gives that
\begin{eqnarray}\label{thmABABnew2eq1}
  L(z,f)
  &=&\sum_{j=0}^{m}b_{j}(z)\left(\left(-\frac{\pi(z+c_{j})}
  {\alpha(z+c_{j})}\right)^{1/n}
  e^{\frac{g(z+c_{j})}{n}}\right)^{(k_{j})}\nonumber \\
  &=& \sum_{j=0}^{m}b_{j}(z)\gamma(z+c_{j}) e^{\frac{g(z+c_{j})}{n}}\nonumber \\
  &=& \left(\sum_{j=0}^{m}b_{j}(z)\gamma(z+c_{j}) e^{\frac{g(z+c_{j})-g(z)}{n}}\right)e^{\frac{g(z)}{n}}
  =\beta,
\end{eqnarray}
where $\gamma$ is a differential polynomial of $\left(-\frac{\pi}
  {\alpha}\right)^{1/n}$ and $g$. Obviously, by Lemma~\ref{lemmaxza},
 $T(r,\gamma(z+c_{j}))=T(r,\gamma(z))+S(r,\gamma(z))=S_{\lambda}(r,f)$.

 If $\sum_{j=0}^{m}b_{j}(z)\gamma(z+c_{j})e^{\frac{g(z+c_{j})-g(z)}{n}}\equiv 0$, then we have
 $L(z,f)\equiv\beta\equiv 0$, a  contradiction with the assumption that $L(z,f)\not\equiv 0$.

 If $\sum_{j=0}^{m}b_{j}(z)\gamma(z+c_{j})
 e^{\frac{g(z+c_{j})-g(z)}{n}}\not\equiv 0$,
next we  prove  that
any $\lambda$-small function of $f$ is also a small function of $e^{g}$.
From  \eqref{th3eq3add2lem}  and \eqref{th3eq1},  we have
\begin{eqnarray}\label{th3eq3add2}
  T(r,e^{g})&=&T\left(r,\frac{1}{\pi}(\Phi_{n}-\beta)\right)
  \geq T(r,\Phi_{n}-\beta)-T(r,\pi)\nonumber \\
  &=& T(r,\Phi_{n}-\beta)-S_{\lambda}(r,f)\nonumber  \\
  &\geq& (n-1)T(r,f)-S_{\lambda}(r,f) \nonumber  \\
&=&(n-1-o(1))T(r,f)- O(r^{\lambda+\varepsilon}).
\end{eqnarray}

By applying the exponential polynomial theory (see \cite[Lemma 2.6]{Li-Yang} or \cite{Steinmetz}), we have
\begin{eqnarray}\label{th3eq400666}
  T(r,e^{g})=Ar^{\rho}+o(r^{\rho}),
\end{eqnarray}
where $A$ is a non-zero constant. Combining \eqref{th3eq3add2} with \eqref{th3eq400666},  we obtain
\begin{eqnarray*}
  \frac{S_{\lambda}(r,f)}{T(r,e^{g})}&=&\frac{O(r^{\lambda+\varepsilon})+S(r,f)}{T(r,e^{g})}
=\frac{O(r^{\lambda+\varepsilon})}{T(r,e^{g})}+\frac{S(r,f)}{T(r,e^{g})}\nonumber \\
&\leq& \frac{O(r^{\lambda+\varepsilon})}{Ar^{\rho}}+\frac{S(r,f)}{T(r,f)}\cdot \frac{2T(r,e^{g})+O(r^{\lambda+\varepsilon})}{T(r,e^{g})}\nonumber \\
&=&   \frac{O(r^{\lambda+\varepsilon})}{Ar^{\rho}}+\frac{S(r,f)}{T(r,f)}\cdot \left(2+\frac{O(r^{\lambda+\varepsilon})}{Ar^{\rho}}\right)\rightarrow 0,
\, \textrm{as}\,
 r\to \infty.
\end{eqnarray*}
 Thus  $S_{\lambda}(r,f)\subseteq S(r,e^{g})$. By combining with
 \eqref{thmABABnew2eq1}, we have
 \begin{eqnarray*}
T\left(r,e^{g(z)}\right)=T\left(r, \frac{\beta^{n}}{\left(\sum_{j=0}^{m}b_{j}(z)\gamma(z+c_{j})
e^{\frac{g(z+c_{j})-g(z)}{n}}\right)^{n}} \right)=S(r,e^{g}),
 \end{eqnarray*}
which yields a contradiction.

Subcase 1.2. $\widetilde{c}\neq-1$. Then from  Lemma~\ref{lemmaxzab} and \eqref{th3eq3}, we get
\begin{eqnarray*}
  nT(r,f)&=&T(r,f^{n})=T\left(r, \frac{L-\beta}{\left(\frac{1}{\widetilde{c}}+1\right)\alpha}  \right)\\
  &\leq& T(r,L)+S_{\lambda}(r,f)\\
  &\leq& m\left(r,\frac{L}{f}\right)+m(r,f)+S_{\lambda}(r,f)\\
  &\leq& T(r,f)+S_{\lambda}(r,f),
\end{eqnarray*}
which yields a contradiction since $n\geq3$.

Case 2. $G(z, f)\not\equiv 0$. Since $n\geq3$, by applying Lemma~\ref{lemma24} to \eqref{th3eq2},  we obtain
\begin{eqnarray*}
  T(r, G(z,f))=m(r,G(z,f))+N(r,G(z,f))
=O(r^{\rho-1+\varepsilon})+S_{\lambda}(r,f),
\end{eqnarray*}
and
\begin{eqnarray*}
  T(r,fG(z,f))=m(r,fG(z,f))+N(r,fG(z,f))=O(r^{\rho-1+\varepsilon})+S_{\lambda}(r,f).
\end{eqnarray*}
Therefore,
\begin{eqnarray*}
  T(r,f)=T\left(r, \frac{fG(z,f)}{G(z,f)}\right)\leq T(r,fG(z,f))+T(r,G(z,f))=O(r^{\rho-1+\varepsilon})+S_{\lambda}(r,f),
\end{eqnarray*}
which is a contradiction. Hence $\lambda(\Phi_{n}-\beta)=\rho(f)$.

Finally, we prove the case $\rho=0$. By Lemma~\ref{lemma000}, we have
$0\leq\lambda(\Phi_{n}-\beta)\leq\sigma(\Phi_{n}-\beta)=\rho=0$. Thus,
$\lambda(\Phi_{n}-\beta)=\rho=0$. Next we prove that $\Phi_{n}-\beta$ has
infinitely many zeros. Suppose contrary to the assertion, that
$\Phi_{n}-\beta$ has finitely many zeros, then by the
standard Hadamard representation and $\rho=0$, we may write
\begin{eqnarray}\label{th3eq1666}
 \Phi_{n}-\beta= L(z,f)-\alpha f^{n}-\beta=\widetilde{\pi},
\end{eqnarray}
where $\widetilde{\pi}$ is a non-vanishing small function of
$f$. Thus, by Lemma~\ref{lemmaxzab} we have
\begin{eqnarray*}
 nT(r,f)=T(r,f^{n})&=&T\left(r,\frac{L(z,f)-\widetilde{\pi}-\beta}{\alpha} \right)\\
 &\leq&  T(r,L(z,f))+S(r,f)\\
&=& m(r,L(z,f))+S(r,f) \\
 &\leq& \sum_{j=0}^{m} m\left(r,  \frac{f^{(k_{j})}(z+c_{j})}{f(z)} \right)+m(r,f)+S(r,f) \\
  &\leq& T(r,f)+S(r,f),
\end{eqnarray*}
leading to a contradiciton by $n\geq 3$. Thus $\Phi_{n}-\beta$ has infinitely many zeros.

\section{Proof of Theorem~\ref{thmABABnew4}.}
(i) Suppose that $d$ is a Borel exceptional value
of $f(z)$, and
\begin{eqnarray*}
\left(\sum_{j\in I_{1}}b_{j}\right) d-\alpha d^{2}-\beta\not\equiv 0.
\end{eqnarray*}
Then $f(z)$ can be written in the form
\begin{eqnarray}\label{thmABABnew3eq1add0}
  f(z)=d+h(z)e^{az^{\rho}},
\end{eqnarray}
where $a\neq 0$ is a constant,   $\rho(\geq 1)$ is an integer, and
 and $h$ is a non-vanishing
entire function such that $\sigma(h)<\rho$.
Thus
\begin{eqnarray}\label{thmABABnew3eq1add}
  f(z+c_{j})&=&d+h(z+c_{j})e^{a(z+c_{j})^{\rho}}\nonumber \\
  &=&d+\left(h(z+c_{j})e^{a(z+c_{j})^{\rho}-az^{\rho}}\right)e^{az^{\rho}}\nonumber \\
  &=& d+h(z+c_{j})\widetilde{h}_{c_{j}}e^{az^{\rho}},
\end{eqnarray}
where  $\widetilde{h}_{c_{j}}=e^{a(z+c_{j})^{\rho}-az^{\rho}}$. Combining with  Lemma~\ref{lemmaxza},  $\sigma(h(z+c_{j})\widetilde{h}_{c_{j}})<\rho$.
For $k_{j}>0$, differentiating iteratively,  we obtain by elementary computation that
\begin{eqnarray}\label{thmABABnew3eq2add}
f^{(k_{j})}(z+c_{j})&=&(d+h(z+c_{j})e^{a(z+c_{j})^{\rho}})^{(k_{j})}\nonumber\\
&=&d^{(k_{j})}+\left(h(z+c_{j})e^{a(z+c_{j})^{\rho}}\right)^{(k_{j})}\nonumber\\
&=&h_{c_{j},k_{j}}e^{a(z+c_{j})^{\rho}}=h_{c_{j},k_{j}}\widetilde{h}_{c_{j}}e^{az^{\rho}},
\end{eqnarray}
where $h_{c_{j},k_{j}}$ are differential polynomials in $h(z+c_{j})$ and $a(z+c_{j})^{\rho}$ . Obviously, $\sigma(h_{c_{j},k_{j}}\widetilde{h}_{c_{j}})<\rho$.
On the other hand, we may write $L(z,f)$ as
\begin{eqnarray}\label{thmABABnew3eq2add900}
L(z,f)
= \sum_{j\in I_{1}}b_{j}(z)f(z+c_{j})+\sum_{j\in I_{2}}b_{j}(z)f^{(k_{j})}(z+c_{j})
\end{eqnarray}
where $I_{1}=\{0\leq j\leq m: k_{j}=0\}$ and $I_{2}=\{0\leq j\leq m: k_{j}>0\}$. Thus, by substituting  \eqref{thmABABnew3eq1add} and  \eqref{thmABABnew3eq2add} into \eqref{thmABABnew3eq2add900},
we obtain
\begin{eqnarray}\label{thmABABnew3eq2add920}
L(z,f)
&=& \sum_{j\in I_{1}}b_{j}(z)\left( d+h(z+c_{j})\widetilde{h}_{c_{j}}e^{az^{\rho}} \right)+\sum_{j\in I_{2}}b_{j}(z)h_{c_{j},k_{j}}\widetilde{h}_{c_{j}}e^{az^{\rho}}\nonumber\\
&=& \left(\sum_{j\in I_{1}}b_{j}\right)d+
\left( \sum_{j\in I_{1}}b_{j}h(z+c_{j})\widetilde{h}_{c_{j}}+
 \sum_{j\in I_{2}}b_{j}h_{c_{j},k_{j}}\widetilde{h}_{c_{j}}\right)e^{az^{\rho}}
\end{eqnarray}
By combining with  \eqref{thmABABnew3eq1add0} we get
\begin{eqnarray}\label{thmABABnew3eq2add901}
  \Phi_{2}&=&L(z,f)-\alpha f^{2}\nonumber\\
&=&\left(\sum_{j\in I_{1}}b_{j}\right)d+
\left( \sum_{j\in I_{1}}b_{j}h(z+c_{j})\widetilde{h}_{c_{j}}+
 \sum_{j\in I_{2}}b_{j}h_{c_{j},k_{j}}\widetilde{h}_{c_{j}}\right)e^{az^{\rho}}
-\alpha (d+he^{az^{\rho}})^{2}\nonumber\\
&=& \widetilde{\gamma}(z)
e^{az^{\rho}} -\alpha h^{2}e^{2az^{\rho}}
+\left(\sum_{j\in I_{1}}b_{j}\right) d-\alpha d^{2},
\end{eqnarray}
where
\begin{eqnarray*}
  \widetilde{\gamma}(z)=\sum_{j\in I_{1}}b_{j}h(z+c_{j})\widetilde{h}_{c_{j}}+\sum_{j\in I_{2}}b_{j}h_{c_{j},k_{j}}\widetilde{h}_{c_{j}}-2\alpha d h.
\end{eqnarray*}

By Lemma~\ref{lemma000}, $\sigma(\Phi_{2}-\beta)=\rho$.  If $\lambda(\Phi_{2}-\beta)<\rho=\sigma(\Phi_{2}-\beta)$, then $\beta$ is
a Borel exceptional small function of $\Phi_{2}$, and
we can rewrite $\Phi_{2}$ as follow:
\begin{eqnarray}\label{thmABABnew3eq2add90200}
 \Phi_{2}=\beta+h^{*}(z)e^{b z^{\rho}},
\end{eqnarray}
where $b(\neq 0)$ is a constant, and $h^{*}(\not\equiv 0)$ is
an entire function with $\sigma(h^{*})<\rho$.   By \eqref{thmABABnew3eq2add901} and \eqref{thmABABnew3eq2add90200} we have
\begin{eqnarray}\label{thmABABnew3eq2add903}
  h^{*}(z)e^{b z^{\rho}}=\widetilde{\gamma}(z)
e^{az^{\rho}} -\alpha h^{2}e^{2az^{\rho}}
+\left(\sum_{j\in I_{1}}b_{j}\right) d-\alpha d^{2}-\beta.
\end{eqnarray}

In \eqref{thmABABnew3eq2add903}, there are three cases for $b$: Case 1. $b\neq a$ and $b\neq 2a$; Case 2. $b=a$; Case 3. $b=2a$.

Applying Lemma~\ref{lemma21} to \eqref{thmABABnew3eq2add903} for all
these three cases, we obtain
\begin{eqnarray*}
  \left(\sum_{j\in I_{1}}b_{j}\right) d-\alpha d^{2}-\beta\equiv0,
\end{eqnarray*}
which contradicts our assumption that
\begin{eqnarray*}
\beta\not\equiv\left(\sum_{j\in I_{1}}b_{j}\right) d-\alpha d^{2}.
\end{eqnarray*}
Hence, $\lambda(\Phi_{2}(z)-\beta)=\rho$.

(ii) Suppose that $d$ is a Borel exceptional value
of $f$, and
\begin{eqnarray}\label{thmABABnew3eq2add901aab}
\left(\sum_{j\in I_{1}}b_{j}\right) d-\alpha d^{2}-\beta\equiv 0.
\end{eqnarray}
 Using
the same method as before, we can obtain  \eqref{thmABABnew3eq1add0}, \eqref{thmABABnew3eq2add920} and \eqref{thmABABnew3eq2add901}.
By combining \eqref{thmABABnew3eq2add901} with \eqref{thmABABnew3eq2add901aab}, we have
\begin{eqnarray}\label{thmABABnew3eq2add901aaa}
  \Phi_{2}-\beta
= \widetilde{\gamma}(z)e^{az^{\rho}} -\alpha h^{2}e^{2az^{\rho}},
\end{eqnarray}

Next, we discuss the following two cases:

Case 1.  $\widetilde{\gamma}(z)\equiv 0$. Then
\begin{eqnarray}\label{thmABABnew3eq2add901kkk}
 \sum_{j\in I_{1}}b_{j}h(z+c_{j})\widetilde{h}_{c_{j}}+\sum_{j\in I_{2}}b_{j}h_{c_{j},k_{j}}\widetilde{h}_{c_{j}}\equiv2\alpha d h,
\end{eqnarray}
and
\eqref{thmABABnew3eq2add901aaa} can be reduced to
\begin{eqnarray}\label{thmABABnew3eq2add901aac}
  \Phi_{2}-\beta=-\alpha h^{2}e^{2az^{\rho}}.
\end{eqnarray}
By Lemma~\ref{lemma000}, $\sigma(\Phi_{2}-\beta)=\rho$. Combining with
$\sigma(h)<\rho$, we obtain $\lambda(\Phi_{2}-\beta)=\lambda( h^{2})\leq\sigma( h)<\rho$. Thus,  $\beta$ is a Borel exceptional small function of $\Phi_{2}$.

 From \eqref{thmABABnew3eq1add0} and \eqref{thmABABnew3eq2add901aac}, we have
\begin{eqnarray}\label{thmABABnew3eq2add907}
  \Phi_{2}&=&\beta-\alpha \left(he^{a z^{\rho}}\right)^{2}=
\beta-\alpha(f-d)^{2}
\end{eqnarray}
Hence
\begin{eqnarray*}
 \frac{\beta-\Phi_{2}}{(f-d)^{2}}=\alpha.
\end{eqnarray*}
Combining with \eqref{thmABABnew3eq1add0}, \eqref{thmABABnew3eq2add920},   \eqref{thmABABnew3eq2add901aab} and \eqref{thmABABnew3eq2add901kkk},
 we obtain
\begin{eqnarray*}
  L(z,f)&=&\left(\sum_{j\in I_{1}}b_{j}\right)d+
\left( \sum_{j\in I_{1}}b_{j}h(z+c_{j})\widetilde{h}_{c_{j}}+
 \sum_{j\in I_{2}}b_{j}h_{c_{j},k_{j}}\widetilde{h}_{c_{j}}\right)e^{az^{\rho}}\nonumber \\
&=& \alpha d^{2}+\beta+2\alpha d he^{az^{\rho}} \nonumber \\
&=& \alpha d^{2}+\beta+2\alpha d(f-d).
\end{eqnarray*}
Therefore,
\begin{eqnarray*}
  \frac{ L(z,f)-\alpha d^{2}-\beta}{2d(f-d)} =\alpha.
\end{eqnarray*}

Case 2.  $\widetilde{\gamma}(z)\not\equiv 0$.
 We
rewrite \eqref{thmABABnew3eq2add901aaa} as follow:
\begin{eqnarray}\label{thmABABnew3eq2add901aae}
  \Phi_{2}-\beta= \widetilde{\gamma}
e^{az^{\rho}} -\alpha h^{2}e^{2az^{\rho}}
=\alpha h^{2}e^{az^{\rho}}\left(\frac{\widetilde{\gamma}}{\alpha h^{2}}-e^{az^{\rho}}\right).
\end{eqnarray}
Next, we prove that $T(r, \widetilde{\gamma}/(\alpha h^{2}))=S(r,e^{az^{\rho}})$. By $\sigma(h)<\rho$, we have $T(r,h)=S(r,e^{az^{\rho}})$. Combining with Lemma~\ref{lemmaxza}, we have $T(r,h(z+c_{j}))=S(r,e^{az^{\rho}})$ and
$T(r, h_{c_{j},k_{j}})=S(r,e^{az^{\rho}})$.
We assert that  $S_{\lambda}(r,f)\subseteq S(r,e^{az^{\rho}})$. From \eqref{thmABABnew3eq1add0}, we have
\begin{eqnarray}\label{thmABABnew3eq2add901aag}
T(r,f)=T(r,e^{az^{\rho}})+S(r,e^{az^{\rho}}).
\end{eqnarray}
Thus,
\begin{eqnarray*}
  \frac{S_{\lambda}(r,f)}{T(r,e^{az^{\rho}})}&=&
\frac{O(r^{\lambda+\varepsilon})}{T(r,e^{az^{\rho}})}+\frac{S(r,f)}{T(r,e^{az^{\rho}})}\nonumber =\frac{O(r^{\lambda+\varepsilon})}{\frac{|a|}{\pi}r^{\rho}}+
\frac{S(r,f)}{T(r,f)}\to 0, \,\textrm{as}\, r\to \infty.
\end{eqnarray*}
So we have $T(r,b_{j})=S(r,e^{az^{\rho}})$.  Thus,
$T(r,\widetilde{\gamma}/(\alpha h^{2}))=S(r,e^{az^{\rho}})$.

By the first and second main theorems of Nevanlinna theory, we have
\begin{eqnarray*}
 T\left(r, e^{az^{\rho}}\right)
&\leq&
N\left(r,\frac{1}{e^{az^{\rho}}} \right)
+N\left(r,\frac{1}{e^{az^{\rho}}-\frac{\widetilde{\gamma}}{\alpha h^{2}}}  \right)
+N(r,e^{az^{\rho}})+S(r,e^{az^{\rho}})
\nonumber\\
&=& N\left(r,\frac{1}{e^{az^{\rho}}-\frac{\widetilde{\gamma}}{\alpha h^{2}}} \right)
+S(r,e^{az^{\rho}})\nonumber\\
&\leq& T\left(r, e^{az^{\rho}}\right)+S(r,e^{az^{\rho}}).
\end{eqnarray*}
So
\begin{eqnarray}\label{thmABABnew3eq2add901aaf}
 N\left(r,\frac{1}{\Phi_{2}-\beta}\right)= N\left(r,\frac{1}{e^{az^{\rho}}-\frac{\widetilde{\gamma}}{\alpha h^{2}}}  \right)
+S(r,e^{az^{\rho}})
=T\left(r,e^{az^{\rho}}\right)+S(r,e^{az^{\rho}}).
\end{eqnarray}
Thus, combining with \eqref{thmABABnew3eq2add901aag} and \eqref{thmABABnew3eq2add901aaf},
we obtain
\begin{eqnarray*}
  N\left(r,\frac{1}{\Phi_{2}-\beta}\right)=
T(r,f)+S_{\lambda}(r,f).
\end{eqnarray*}

\section{Proof of Theorem~\ref{thmABABnew7}.}

Suppose that $d=0$ is the Borel exceptional value of $f$.
Using
the same method as before (the proof of Theorem~\ref{thmABABnew4}), we can obtain \eqref{thmABABnew3eq2add920} and \eqref{thmABABnew3eq2add901} with $d=0$, i.e.,
\begin{eqnarray}\label{thmABABnew3eq2add800}
L(z,f)=
\widetilde{\gamma}(z)e^{az^{\rho}}.
\end{eqnarray}
and
\begin{eqnarray}\label{thmABABnew3eq2add901zzz}
  \Phi_{2}
= \widetilde{\gamma}(z)
e^{az^{\rho}} -\alpha h^{2}e^{2az^{\rho}},
\end{eqnarray}
where
\begin{eqnarray*}
 \widetilde{\gamma}(z)=\sum_{j\in I_{1}}b_{j}h(z+c_{j})\widetilde{h}_{c_{j}}+\sum_{j\in I_{2}}b_{j}h_{c_{j},k_{j}}\widetilde{h}_{c_{j}}.
\end{eqnarray*}
Next, we discuss the following two cases:

Case 1. $\beta\not\equiv 0$.  By Lemma~\ref{lemma000}, $\sigma(\Phi_{2}-\beta)=\rho$.  If $\lambda(\Phi_{2}(z)-\beta)<\rho$, then
we can rewrite $\Phi_{2}$ as follow:
\begin{eqnarray}\label{thmABABnew3eq2add902}
 \Phi_{2}=\beta+h^{*}e^{b z^{\rho}},
\end{eqnarray}
where $b(\neq 0)$ is a constant, and $h^{*}(\not\equiv 0)$ is
an entire function with $\sigma(h^{*})<\rho$. By \eqref{thmABABnew3eq2add901zzz} and \eqref{thmABABnew3eq2add902} we have
\begin{eqnarray}\label{thmABABnew3eq2add90300zza}
  \beta+h^{*}e^{b z^{\rho}}&=&\widetilde{\gamma}(z)
e^{az^{\rho}} -\alpha h^{2}e^{2az^{\rho}}.
\end{eqnarray}
In \eqref{thmABABnew3eq2add90300zza}, there are three cases for $b$: Case 1. $b\neq a$ and $b\neq 2a$; Case 2. $b=a$; Case 3. $b=2a$. Applying Lemma~\ref{lemma21} to \eqref{thmABABnew3eq2add90300zza} for
these three cases, we obtain $\beta\equiv 0$, which contradicts our
assumption that $\beta\not\equiv 0$. Hence $\lambda(\Phi_{2}(z)-\beta)=\rho$.

 Case 2. $\beta\equiv 0$. Obviously, $\widetilde{\gamma}(z) \not\equiv 0$. Otherwise,
by \eqref{thmABABnew3eq2add800} we obtain $L(z,f)\equiv 0$, a contradiction.
We rewrite \eqref{thmABABnew3eq2add901zzz} as follow:
\begin{eqnarray*}
  \Phi_{2}
= \alpha h^{2}e^{az^{\rho}}\left(\frac{\widetilde{\gamma}}{\alpha h^{2}}-e^{az^{\rho}} \right),
\end{eqnarray*}
following the same method as in the proof of case 2 in Theorem~\ref{thmABABnew4}(ii),
 we obtain
\begin{eqnarray*}
N\left(r,\frac{1}{\Phi_{2}}\right)&=&N\left(r,\frac{1}{e^{az^{\rho}}-\frac{\widetilde{\gamma}}{\alpha h^{2}}}  \right)
+S(r,e^{az^{\rho}})
=T\left(r,e^{az^{\rho}}\right)+S(r,e^{az^{\rho}})\\
&=&T(r,f)+S_{\lambda}(r,f).
\end{eqnarray*}
Hence, we have $\lambda(\Phi_{2})=\rho$.

\section*{Acknowledgments}
This work has no conflicts of interest. The author would like to thank the referees for a careful reading of the manuscript
and valuable comments! This work was supported by  NNSF of China (No.11801215), and the NSF of Shandong Province, P. R. China (No. ZR2018MA021).

%\smallskip

\section*{References}

%\bibliographystyle{elsarticle-num}
%\bibliography{database0110}

\def\cprime{$'$}

\end{document}